\documentclass[reqno, 12pt]{article}

\pdfoutput=1

\usepackage{enumerate}
\usepackage{latexsym}
\usepackage[centertags]{amsmath}
\usepackage{amsfonts}
\usepackage{amssymb}
\usepackage{amsthm}
\usepackage{mathtools}
\usepackage{newlfont}
\usepackage{graphics}
\usepackage{color}
\usepackage{float}
\usepackage{diagbox}
\usepackage{tocloft}
\usepackage{titlesec}
\usepackage{booktabs}
\usepackage{extpfeil}
\usepackage{centernot}
\usepackage[pagebackref]{hyperref}
\usepackage[linesnumbered,ruled,vlined]{algorithm2e}
\usepackage{url}
\usepackage{hyperref}
\usepackage{longtable}
\usepackage{rotating}
\usepackage{multirow}
\usepackage{extarrows}
\usepackage[sort,compress,numbers]{natbib}
\usepackage[utf8]{inputenc}

\newtheorem{thm}{Theorem}[section]

\newtheorem{lem}[thm]{Lemma}
\newtheorem{prop}[thm]{Proposition}

\theoremstyle{mydefinition}
\newtheorem{dfn}[thm]{Definition}
\theoremstyle{myremark}
\newtheorem{rem}[thm]{Remark}
\newtheorem{exa}[thm]{Example}

\newtheorem{prob}[thm]{Open Problem}

\allowdisplaybreaks[4]

\SetKwInput{KwInput}{Input}                % Set the Input
\SetKwInput{KwOutput}{Output}              % set the Output

\renewcommand{\P}{{\mathbb{P}}}

\title{Ehrhart Theory of the Join of Two Lattice Polytopes}

\author{Feihu Liu$^{\color{blue} \dag}$, Sihao Tao$^{\color{blue} \S}$, and Guoce Xin$^{\color{blue} \P}$
\\[2mm]
{\small $^{\color{blue} \dag, \S, \P}$ School of Mathematical Sciences,}\\[-0.8ex]
{\small Capital Normal University, Beijing, 100048, P.R.~China}\\
{\small {\color{blue} $^\dag$} Email address: liufeihu7476@163.com}\\
{\small {\color{blue} $^\S$} Email address: sihao\_tao@cnu.edu.cn}\\
{\small {\color{blue} $^\P$} Email address: guoce\_xin@163.com}
}

\date{\today}

\begin{document}

\maketitle

\begin{abstract}
Inspired by research on the Cartesian product of two lattice polytopes, this paper investigates the Ehrhart theory of the join of two lattice polytopes. This is also a well-known open problem listed on the website of the American Institute of Mathematics. This paper resolves this open problem. We first construct counterexamples showing that the join of two Ehrhart positive polytopes is not necessarily Ehrhart positive. Then we prove that if two lattice polytopes have the integer decomposition property and the spanning property, then their join also has these two properties. However, the very ample property is not inherited under joins. Finally, we show that unimodular triangulations, regular triangulations, and quadratic triangulations are preserved under the join operation. As a byproduct, we state the necessary and sufficient condition for the Cartesian product of two Gorenstein lattice polytopes to remain Gorenstein.
\end{abstract}

\noindent
\begin{small}
\emph{2020 Mathematics subject classification}: Primary 52B20;  Secondary 52B05; 52B11; 05A15.
\end{small}

\noindent
\begin{small}
\emph{Keywords}: Lattice polytope; Ehrhart positive; Gorenstein polytope; Join; Integer decomposition property; Unimodular triangulation; Regular triangulation; Quadratic triangulation.
\end{small}

\tableofcontents

\section{Introduction}
Let $\mathcal{P}$ be a \emph{lattice $d$-polytope}, i.e., a convex polytope in $\mathbb{R}^n$ whose vertices are elements of $\mathbb{Z}^n$ and whose affine span has dimension $d\leq n$.
Lattice polytopes are also called \emph{integral polytopes}.
For any nonnegative integer $t$, let $t\mathcal{P} = \{ t\alpha : \alpha \in \mathcal{P} \}$ be the $t$-th dilation of $\mathcal{P}$. Then consider the function
$$i(\mathcal{P},t)=|t\mathcal{P}\cap \mathbb{Z}^n|,\quad t=1,2,\ldots,$$
which counts the number of integer points in $t\mathcal{P}$.

The study on $i(\mathcal{P},t)$ originated with Ehrhart~\cite{Ehrhart62} in 1962. He proved that the function $i(\mathcal{P},t)$ is a polynomial in $t$ of degree $d$ with constant term $1$.
Now $i(\mathcal{P},t)$ is called the \emph{Ehrhart polynomial} of $\mathcal{P}$.
A basic result~\cite[Corollary 3.20; Theorem 5.6]{BeckRobins} is that the leading coefficient of $i(\mathcal{P},t)$ equals the relative volume of $\mathcal{P}$ and the second highest coefficient of $i(\mathcal{P},t)$ equals half of the boundary volume of $\mathcal{P}$.
The remaining coefficients, however, are more intricate and lack a straightforward geometric interpretation~\cite{McMullen77}.

The generating function
$$\mathrm{Ehr}(\mathcal{P},x)=1+\sum_{t\geq 1} i(\mathcal{P},t)x^t$$
is called the \emph{Ehrhart series} of $\mathcal{P}$.
It can be expressed in the form
$$\mathrm{Ehr}(\mathcal{P},x)=\frac{h_{\mathcal{P}}^{*}(x)}{(1-x)^{d+1}}=\frac{h_d^{*}x^d + h_{d-1}^{*}x^{d-1}+ \cdots+ h_1^*x+ h_0^*}{(1-x)^{d+1}},$$
where $\dim\mathcal{P}=d$ and $h^{*}(x)$ is a polynomial in $x$ of degree at most $d$ (see~\cite[Chapter 3.5]{BeckRobins}).
The polynomial $h_{\mathcal{P}}^*(x)$ is commonly known as the \emph{$h^*$-polynomial} of $\mathcal{P}$.
Stanley~\cite{Stanleyh-polynomial} (or~\cite[Theorem 3.12]{BeckRobins}) shows that the coefficients of $h^*$-polynomial are nonnegative integers.
Clearly the Ehrhart polynomial is given by
\begin{align*}
i(\mathcal{P},t) = \sum_{j=0}^{d} h^*_j \binom{t + d - j}{d}.
\end{align*}

For further background and knowledge on polytopes, we refer to several excellent books~\cite{BeckRobins}, \cite[Chapter 4]{RP.Stanley}, and \cite{Ziegler}, as well as the classic survey articles \cite{Ferroni23} and \cite{FuLiu19}.
The main object of study in this paper is the join of two polytopes. The definition appears in~\cite{Henk2004}.

\begin{dfn}
Let $\mathcal{P} \subset \mathbb{R}^p$ and $\mathcal{Q} \subset \mathbb{R}^q$ be two polytopes. The \emph{join} of $\mathcal{P}$ and $\mathcal{Q}$, denoted by $\mathcal{P} * \mathcal{Q}$, is defined as the convex hull of their respective embeddings into orthogonal subspaces of $\mathbb{R}^{p+q+1}$, separated by a height of $1$ in the final coordinate. Formally, it is given by
\begin{align*}
\mathcal{P} * \mathcal{Q} = \operatorname{conv} \left( \{ (x, \mathbf{0}_q, 1) \mid x \in \mathcal{P} \} \cup \{ (\mathbf{0}_p, y, 0) \mid y \in \mathcal{Q} \} \right) \subset \mathbb{R}^{p+q+1},
\end{align*}
where $\mathbf{0}_p$ and $\mathbf{0}_q$ denote the origin in $\mathbb{R}^p$ and $\mathbb{R}^q$, respectively.
\end{dfn}

For instance, the join of two line segments placed in general position yields a tetrahedron.  A pyramid over a polytope $\mathcal{P}$ is defined as the join of $\mathcal{P}$ with a point.
The geometric properties of the join are immediate from affine independence. In particular, the following result holds.

\begin{prop}\label{Prop-Join-Properties}
Let $\mathcal{P} \subset \mathbb{R}^p$ and $\mathcal{Q} \subset \mathbb{R}^q$ be a lattice $p$-polytope and a lattice $q$-polytope, respectively. Then $\mathcal{P} * \mathcal{Q} \subset \mathbb{R}^{p+q+1}$ satisfies the following properties:
\begin{enumerate}
\item[\textup{(i)}] $\mathcal{P} * \mathcal{Q}$ is a lattice $(p+q+1)$-polytope.

\item[\textup{(ii)}] The vertex set of $\mathcal{P} * \mathcal{Q}$ is precisely the disjoint union of the canonical embeddings of the vertex sets of $\mathcal{P}$ and $\mathcal{Q}$, that is,
    $$\operatorname{vert}(\mathcal{P} * \mathcal{Q}) = \left( \operatorname{vert}(\mathcal{P}) \times \{\mathbf{0}_q\} \times \{1\} \right) \cup \left( \{\mathbf{0}_p\} \times \operatorname{vert}(\mathcal{Q}) \times \{0\} \right).$$

\item[\textup{(iii)}] The join of any two lattice simplices is again a lattice simplex.

\item[\textup{(iv)}] The Ehrhart series of the join is the product of their respective Ehrhart series (see {\em\cite[Lemma 1.3]{HenkTagami}} or {\em\cite[Exercise 3.38]{BeckRobins}}), i.e.,
    $$\mathrm{Ehr}(\mathcal{P}*\mathcal{Q},x) = \mathrm{Ehr}(\mathcal{P}, x)\cdot \mathrm{Ehr}(\mathcal{Q},x).$$
\end{enumerate}
\end{prop}

Nill and Schepers~\cite{BNill-Schepers} studied different kinds of joins of Gorenstein polytopes.
Ehrenborg and Fox~\cite{Ehrenborg} determined explicit recurrences for computing the $cd$-index of the join of two polytopes.
Some algebraic applications of the join of two polytopes are given in~\cite{BrunsBinomial,Erlandsson}

A definition similar to the join of polytopes is the \emph{free sum} of polytopes.
The \emph{free sum} of $\mathcal{P}$ and $\mathcal{Q}$, denoted by $\mathcal{P} \oplus \mathcal{Q}$, is defined as the convex hull
\begin{align*}
\mathcal{P} \oplus \mathcal{Q} = \operatorname{conv} \left( \{ (x, \mathbf{0}_q) \mid x \in \mathcal{P} \} \cup \{ (\mathbf{0}_p, y) \mid y \in \mathcal{Q} \} \right) \subset \mathbb{R}^{p+q}.
\end{align*}
The free sum of polytopes has been extensively studied; see, e.g., \cite{BeckJayawant,Braun2006,Stapledon11}.

A well-known open problem concerning the join of two lattice polytopes is listed on the website of the American Institute of Mathematics~\cite{AimPl-Website}.

\begin{prob}{\em \cite[Problem 1.35]{AimPl-Website}}\label{Problem-Join-open}
For lattice polytopes $\mathcal{P}$ and $\mathcal{Q}$, which Ehrhart-theoretic properties pass through their join? For example, if $\mathcal{P}$ and $\mathcal{Q}$ are both Ehrhart positive, is $\mathcal{P} * \mathcal{Q}$ as well?
\end{prob}

In this paper, we resolve this open problem.
The join of two lattice polytopes is studied by comparing it with the relevant results on their Cartesian product.
More precisely, we obtain the following results.
\begin{enumerate}
    \item[\textup{(i)}] If $\mathcal{P}$ and $\mathcal{Q}$ are Ehrhart positive, then $\mathcal{P} * \mathcal{Q}$ is not necessarily Ehrhart positive. (See Theorem \ref{thm:join_not_preserve_positivity} or Example \ref{Example-Join-Ehrhartpositive}.)

    \item[\textup{(ii)}] If $h^*_{\mathcal{P}}$ and $h^*_{\mathcal{Q}}$ are unimodal, then $h^*_{\mathcal{P} *\mathcal{Q}}$ is not necessarily unimodal. (See Theorem \ref{Theorem-Unimod-join}.)

    \item[\textup{(iii)}] If $\mathcal{P}$ and $\mathcal{Q}$ are spanning, then $\mathcal{P} * \mathcal{Q}$ is spanning. (See Theorem \ref{Theorem-Spanning-Join}.)

    \item[\textup{(iv)}] If $\mathcal{P}$ and $\mathcal{Q}$ are very ample, then $\mathcal{P} * \mathcal{Q}$ is not necessarily very ample. (See Theorem \ref{Theorem-Veryample-Join}.)

    \item[\textup{(v)}] If $\mathcal{P}$ and $\mathcal{Q}$ have integer decomposition property, then $\mathcal{P} * \mathcal{Q}$ has integer decomposition property. (See Theorem \ref{Theorem-IDP-Join}.)

    \item[\textup{(vi)}] If $\mathcal{P}$ and $\mathcal{Q}$ have unimodular triangulations, then $\mathcal{P} * \mathcal{Q}$ has a unimodular triangulation. (See Theorem \ref{Theorem-Unimodular-Triangulation}.)

    \item[\textup{(vii)}] If $\mathcal{P}$ and $\mathcal{Q}$ have regular unimodular triangulations, then $\mathcal{P} * \mathcal{Q}$ has a regular unimodular triangulation. (See Theorem \ref{Theorem-Regular-Uniom-Triangu}.)

    \item[\textup{(viii)}] If $\mathcal{P}$ and $\mathcal{Q}$ have quadratic triangulations, then $\mathcal{P} * \mathcal{Q}$ has a quadratic triangulation. (See Theorem \ref{Theorem-Quadratic-Triang-Join}.)
\end{enumerate}

As a byproduct, we also prove that (also see Theorem \ref{Theorem-Gorenstein-Cartesion}):
If $\mathcal{P}$ and $\mathcal{Q}$ are Gorenstein lattice polytopes with indices $r_{\mathcal{P}}$ and $r_{\mathcal{Q}}$, respectively, then the Cartesian product $\mathcal{P} \times \mathcal{Q}$ is a Gorenstein polytope (i.e., $h^*_{\mathcal{P} \times \mathcal{Q}}(x)$ is palindromic) if and only if $r_{\mathcal{P}} = r_{\mathcal{Q}}$.

The paper is organized as follows.
In Section \ref{Section-Two-join}, we discuss the Ehrhart positivity of the join of two lattice polytopes. Specifically, we construct counterexamples showing that the join of two Ehrhart positive polytopes is not necessarily Ehrhart positive.
Section \ref{Section-2-Preliminary} presents results on the Cartesian product of two polytopes, providing a contrast to our study of the join operation.
In Section \ref{Section-Four-unimod}, we first state a necessary and sufficient condition for the Cartesian product of two Gorenstein lattice polytopes to remain Gorenstein. We then provide a counterexample demonstrating that the join of two $h^*$-unimodal polytopes is not necessarily $h^*$-unimodal.
Section \ref{Section-Five-idp} is divided into three subsections, which respectively explore the integer decomposition property, very ampleness, and the spanning property of the join of two lattice polytopes.
Section \ref{Section-Six-Trign} focuses on triangulations of the join of two lattice polytopes, with particular emphasis on unimodular, regular, and quadratic triangulations.
Finally, Section \ref{Section-Seven-CR} provides concluding remarks.

\section{Ehrhart positivity}\label{Section-Two-join}

This section focuses on the Ehrhart positivity of the join of two lattice polytopes.
This is a subproblem primarily mentioned in Open problem \ref{Problem-Join-open}.
This section is relatively independent of the other sections.

A  lattice polytope $\mathcal{P}$ is said to be \emph{Ehrhart positive} (or to possess \emph{Ehrhart positivity}) if all coefficients of $i(\mathcal{P},t)$ are non-negative.
For a comprehensive introduction to Ehrhart positivity, we refer the reader to the survey by Liu~\cite{FuLiu19}.

The \emph{Bernoulli numbers of the second kind} $\mathcal{B}_n$ are defined via the generating function
$$\frac{x}{\exp(x)-1}=\sum_{n=0}^{\infty} \mathcal{B}_n \frac{(-1)^n x^n}{n!}.$$
The first few values of $\mathcal{B}_n$ are as follows:
$$\mathcal{B}_0=1,\quad \mathcal{B}_1=\frac{1}{2},\quad \mathcal{B}_2=\frac{1}{6},\quad \mathcal{B}_4=-\frac{1}{30},\quad \mathcal{B}_6=\frac{1}{42},$$
and $\mathcal{B}_n=0$ for all odd integers $n\geq 3$.

Let $\mathfrak{S}_d$ denote the symmetric group on $d$ elements. For a permutation $\pi \in \mathfrak{S}_d$, a \emph{descent} is an index $j \in \{1, \dots, d-1\}$ such that $\pi_j > \pi_{j+1}$. We denote by $\mathrm{des}(\pi)$ the number of descents of $\pi$. The \emph{Eulerian number} $A(d,i)$ counts the number of permutations in $\mathfrak{S}_d$ with exactly $i-1$ descents. The $d$-th \emph{Eulerian polynomial} is then defined as
$$A_d(z) = \sum_{i=1}^d A(d,i) z^i = \sum_{\pi \in \mathfrak{S}_d} z^{1+\mathrm{des}(\pi)}.$$
The polynomial $g(t) = \sum_{j=0}^{d+2} C_j t^{d+2-j}$ is defined via the generating function
$$\sum_{t=0}^{\infty} g(t) z^t = \frac{A_d(z)}{(1-z)^{d+3}}.$$

\begin{lem}\label{lem:coefficient_formula}
Let $d \in \mathbb{Z}$ with $d \ge 4$.
Then for $2 \le j \le d$, the coefficient $C_j$ is given by
$$C_j = \frac{1}{(d+1)(d+2)} \binom{d+2}{j} \Big( (1-j)\mathcal{B}_j + j \mathcal{B}_{j-1} \Big).$$
Furthermore, the remaining coefficients are
$$C_0=\frac{1}{(d+1)(d+2)},\quad C_1=\frac{1}{d+1},\quad C_{d+1} = \mathcal{B}_d - \mathcal{B}_{d+1},\quad \text{and}\quad C_{d+2}=0.$$
\end{lem}
\begin{proof}
The classical identity for the Eulerian polynomial (e.g., see~\cite[Proposition 1.1.4]{RP.Stanley}) is
$$\sum_{t=0}^{\infty} t^d z^t = \frac{A_d(z)}{(1-z)^{d+1}}.$$
Multiplying the generating function by $\frac{1}{(1-z)^2}$ corresponds to taking the second-order partial sums of the coefficients. Thus, $g(t)$ is given by
$$g(t) = \sum_{j=0}^{t} \sum_{k=0}^{j} k^d = \sum_{k=0}^{t} (t-k+1)k^d = (t+1)\sum_{k=0}^{t} k^d - \sum_{k=0}^{t} k^{d+1}.$$
It is clear that $C_{d+2}=g(0)=0$.
The classical Faulhaber's formula (e,g., see~\cite[Chapter 4]{Book-Numbers}) states
$$\sum_{k=0}^{t} k^m = \frac{1}{m+1} \sum_{i=0}^{m} \binom{m+1}{i} \mathcal{B}_i t^{m+1-i}.$$
Substituting this into the expression for $g(t)$, we obtain
$$g(t) = \frac{t+1}{d+1} \sum_{i=0}^{d} \binom{d+1}{i} \mathcal{B}_i t^{d+1-i} - \frac{1}{d+2} \sum_{i=0}^{d+1} \binom{d+2}{i} \mathcal{B}_i t^{d+2-i}.$$
Therefore, we get
$$C_0=\frac{1}{(d+1)(d+2)}\quad \text{and}\quad C_1=\frac{1}{d+1}.$$
We can extract the coefficient $C_j$ for $2 \le j \le d$ by matching the terms of degree $d+2-j$ as
\begin{align*}
C_j &= \frac{1}{d+1} \binom{d+1}{j} \mathcal{B}_j + \frac{1}{d+1} \binom{d+1}{j-1}\mathcal{B}_{j-1} - \frac{1}{d+2} \binom{d+2}{j} \mathcal{B}_j \\
&=  \left(\frac{1}{d+1} \binom{d+1}{j} - \frac{1}{d+2} \binom{d+2}{j} \right)\mathcal{B}_j + \frac{1}{d+1} \binom{d+1}{j-1} \mathcal{B}_{j-1} \\
&= \frac{1}{(d+1)(d+2)} \binom{d+2}{j} \Big( (1-j)\mathcal{B}_j + j \mathcal{B}_{j-1} \Big).
\end{align*}
We now find the coefficient $C_{d+1}$ of the linear term $t$. The linear term of $g(t)$ arises exactly from $i=d$ in the first summation and $i=d+1$ in the second summation. Thus, we obtain
\begin{align*}
C_{d+1} = \frac{1}{d+1} \binom{d+1}{d} \mathcal{B}_d - \frac{1}{d+2} \binom{d+2}{d+1} \mathcal{B}_{d+1} = \mathcal{B}_d - \mathcal{B}_{d+1}.
\end{align*}
The conclusion follows.
\end{proof}

\begin{lem}\label{neg-terms}
The coefficient $C_j$ is negative if and only if $j$ satisfies one of the following conditions:
\begin{enumerate}
\item[\textup{(i)}] $j \equiv 1 \pmod 4$ with $j \ge 5$;
\item[\textup{(ii)}] $j \equiv 2 \pmod 4$ with $j \ge 6$.
\end{enumerate}
Equivalently, the negative coefficients appear exactly at $j \in \{5, 6, 9, 10, 13, 14, 17, 18, \dots\}$.
\end{lem}
\begin{proof}
Recall the properties of the second Bernoulli numbers $\mathcal{B}_k$ for $k \ge 2$, we know that $\mathcal{B}_k = 0$ for odd $k$, and for an even integer $k$:
$$\mathcal{B}_k < 0 \iff k \equiv 0 \pmod 4 \quad \text{and} \quad \mathcal{B}_k > 0 \iff k \equiv 2 \pmod 4.$$
For $2 \le j \le d$, let $V_j = (1-j)\mathcal{B}_j + j \mathcal{B}_{j-1}$.
Since the binomial factor $\frac{1}{(d+1)(d+2)}\binom{d+2}{j}>0$, the sign of $C_j$ is identical to the sign of $V_j$.
We determine when $V_j < 0$ by considering cases for $j$.

Case 1: $j$ is odd with $j \ge 3$.
We now have $\mathcal{B}_j = 0$. Thus, $V_j = j \mathcal{B}_{j-1}$ and
$$V_j < 0 \iff \mathcal{B}_{j-1} < 0 \iff j-1 \equiv 0 \pmod 4 \iff j \equiv 1 \pmod 4.$$
The first such $j$ is $5$.

Case 2: $j$ is even with $j \ge 4$. Since $j-1 \ge 3$ is odd, $\mathcal{B}_{j-1} = 0$. Thus, $V_j = (1-j)\mathcal{B}_j$. Because $1-j < 0$, we obtain
$$V_j < 0 \iff \mathcal{B}_j > 0 \iff j \equiv 2 \pmod 4.$$
The first such $j$ is $j = 6$. (Note that for $j=2$, $V_2 = -\mathcal{B}_2 + 2\mathcal{B}_1 = 5/6 > 0$).

Case 3: For $j = d+1$, the coefficient of the linear term is $C_{d+1} = \mathcal{B}_d - \mathcal{B}_{d+1}$. A similar parity analysis on $d$ yields the identical modular conditions for $C_{d+1} < 0$.

Therefore, $C_j < 0 \iff j \equiv 1 \pmod 4$ (with $j \ge 5$) or $j \equiv 2 \pmod 4$ (with $j \ge 6$).
\end{proof}

\begin{thm}\label{thm:join_not_preserve_positivity}
Let $\mathcal Q = [0, 1] \subset \mathbb{R}$ be the unit interval. For any $d$-dimensional Ehrhart positive lattice polytope $\mathcal{P}$ with $d \geq 4$, there exists a positive integer $r_0$ such that for all integers $r \geq r_0$, the join $r\mathcal{P} * \mathcal Q$ is not Ehrhart positive. In particular, Ehrhart positivity is not preserved under the join operation.
\end{thm}
\begin{proof}
Let $i(\mathcal{P}, t) = 1 + \sum_{i=1}^d a_i t^i$. Note that the leading coefficient $a_d = d! \cdot \operatorname{vol}(\mathcal{P}) > 0$. The Ehrhart polynomial of $r\mathcal{P}$ is
\begin{equation}\label{equation-i(rp)}
 i(r\mathcal{P}, t) = 1 + \sum_{i=1}^d a_i r^i t^i.
\end{equation}
For the unit interval $\mathcal{Q} = [0,1]$, its Ehrhart polynomial is $L(\mathcal{Q}, t) = t + 1$. It is clear that $\mathcal{Q}$ is Ehrhart positive. Its Ehrhart series is given by $\operatorname{Ehr}(\mathcal{Q}, z) = \frac{1}{(1-z)^2}$. Therefore,
$$\operatorname{Ehr}(r\mathcal{P} * \mathcal Q, z) = \operatorname{Ehr}(r\mathcal{P}, z) \cdot \frac{1}{(1-z)^2}.$$
The Ehrhart polynomial of the join is exactly the second cumulative sum of the Ehrhart polynomial of $r\mathcal{P}$:
$$i(r\mathcal{P} * \mathcal Q, t) = \sum_{k=0}^t \sum_{\ell=0}^k i(r\mathcal{P}, \ell).$$
Substituting Equation~\eqref{equation-i(rp)} into the second cumulative sum, we obtain
$$i(r\mathcal{P} * \mathcal Q, t) = \sum_{k=0}^t \sum_{\ell=0}^k 1 + \sum_{i=1}^d a_i r^i \sum_{k=0}^t \sum_{\ell=0}^k \ell^i.$$
By Lemma~\ref{lem:coefficient_formula}, for each $i$, the double sum $\sum_{k=0}^t \sum_{\ell=0}^k \ell^i$ is a polynomial of degree $i+2$ in $t$. In other words, we get
$$i(r\mathcal{P} * \mathcal Q, t) = \frac{1}{2}t^2+\frac{3}{2}t+1 + \sum_{i=1}^d a_i r^i \sum_{j=0}^{i+2} C_j t^{d+2-j},$$
where $C_j$ are the coefficients defined in Lemma~\ref{lem:coefficient_formula}.

We now analyze the asymptotic behavior of the coefficients as $r \to \infty$. The term with the highest power of $r$ comes from the leading term $a_d r^d t^d$ of $i(r\mathcal{P}, t)$. For any degree $m$ ($0 \leq m \leq d+2$), the coefficient of $t^m$ in $i(r\mathcal{P} * \mathcal Q, t)$ has the form
$$C_{d+2-m} \cdot a_d r^d + \mathcal{O}(r^{d-1}).$$
The key observation is that by Lemma~\ref{neg-terms}, for all $d \ge 4$, at least one of $C_{d-1}, C_d$, or $C_{d+1}$ is negative. Specifically:

If $d \equiv 1 \text{ or } 2 \pmod 4$ ($d \ge 5$), then $C_d < 0$, which means the coefficient of $t^2$ is negative for large $r$.

If $d \equiv 0 \pmod 4$ ($d \ge 4$), then $C_{d+1} < 0$, which means the coefficient of $t^1$ is negative for large $r$.

If $d \equiv 3 \pmod 4$ ($d \ge 7$), then $C_{d-1} < 0$, which means the coefficient of $t^3$ is negative for large $r$.

In all cases, there exists a specific degree $m \in \{1, 2, 3\}$ such that the leading term in $r$ of the $t^m$ coefficient is negative. Thus, there exists $r_0 > 0$ such that for all $r \ge r_0$, this coefficient is negative, and $r\mathcal{P} * \mathcal Q$ is not Ehrhart positive.
\end{proof}

\begin{exa}\label{Example-Join-Ehrhartpositive}
Let $\mathcal{P} = [0,1]^4$ denote the $4$-dimensional unit cube, whose Ehrhart polynomial is given by
$$i(\mathcal{P}, t) = (t+1)^4 =t^4+4t^3+6t^2+ 4t +1.$$
Let $\mathcal{Q} = [0,1]$.
Then the Ehrhart series of the join $100\mathcal{P} * \mathcal{Q}$ is
\begin{align*}
\operatorname{Ehr}(100\mathcal{P} * \mathcal Q, x) &= \frac{1}{(1-x)^2}\cdot\sum_{t=0}^{\infty} (100t+1)^4 x^t\\& =\frac{96059601x^4+1087941196x^3+1111938806x^2+104060396x+1}{(1-x)^7}.
\end{align*}
Therefore, we can explicitly determine the Ehrhart polynomial of the join $100\mathcal{P} * \mathcal{Q}$ as follows:
\begin{align*}
i(100\mathcal{P} * \mathcal{Q}, t) &=\frac{10000000}{3} t^{6}+20200000 t^{5} +\frac{128015000}{3} t^{4} +\frac{105060200}{3} t^{3}
\\ & \quad\quad + \frac{12050401}{2} t^{2} - \frac{19139191}{6} t +1.
\end{align*}
Since the linear coefficient is negative, it immediately follows that this polytope is not Ehrhart positive.
\end{exa}

\section{Preliminary knowledge}\label{Section-2-Preliminary}

This section introduces the definition and related results for the Cartesian product of two polytopes.
We will see that it exhibits many properties analogous to the join of two polytopes.

\begin{dfn}
Let $\mathcal{P}\subseteq\mathbb{R}^{p}$ and $\mathcal{Q}\subseteq\mathbb{R}^{q}$. Then the \emph{Cartesian product} of $\mathcal{P}$ and $\mathcal{Q}$ is defined as
$$ \mathcal{P}\times \mathcal{Q} = \left\{(x,y)\in\mathbb{R}^{p+q} \bigm| x\in \mathcal{P},\; y\in \mathcal{Q}\right\}. $$
\end{dfn}

\begin{prop}{\em \cite[Lemma 2.3]{Hibi-Higashitani-Tsuchiya-Yoshida}}\label{prop_Cartesian}
Let $\mathcal{P}\subseteq\mathbb{R}^p$ and $\mathcal{Q}\subseteq\mathbb{R}^q$ be lattice polytopes. The Cartesian product $\mathcal{P}\times \mathcal{Q}$ is a lattice polytope of dimension $p+q$.
Its Ehrhart polynomial satisfies
\begin{align}\label{Equation-product-Hibi}
 i(\mathcal{P}\times \mathcal{Q},t) = i(\mathcal{P},t)\cdot i(\mathcal{Q},t).
\end{align}
\end{prop}

Recall that $i(\mathcal{P}^{\circ},t)$ enumerates the lattice points strictly inside $t\mathcal{P}$, and that $(\mathcal{P} \times \mathcal{Q})^\circ = \mathcal{P}^\circ \times \mathcal{Q}^\circ$. Consequently, identity \eqref{Equation-product-Hibi} extends naturally to relative interiors:
\begin{equation}\label{Equation-product-circ}
  i(\mathcal{P}^{\circ}\times \mathcal{Q}^{\circ},t) = i(\mathcal{P}^{\circ},t)\cdot i(\mathcal{Q}^{\circ},t).
\end{equation}

\begin{dfn}
Let $\mathcal{P} \subset \mathbb{R}^p$ be a lattice polytope.
\begin{itemize}
    \item We say that $\mathcal{P}$ possesses the \emph{integer decomposition property} (IDP) if, for any positive integer $k$, every lattice point $z \in k\mathcal{P} \cap \mathbb{Z}^p$ can be expressed as a sum $z = z_1 + \dots + z_k$, where each $z_i \in \mathcal{P} \cap \mathbb{Z}^p$.

    \item The polytope $\mathcal{P}$ is called \emph{very ample} if it satisfies the integer decomposition property for all sufficiently large $k$; that is, for all $k \gg 0$, every lattice point in $k\mathcal{P} \cap \mathbb{Z}^p$ admits such a decomposition into lattice points of $\mathcal{P}$.

    \item We call $\mathcal{P}$ a \emph{spanning polytope} if the integral affine span of its lattice points, denoted by $\operatorname{aff}_{\mathbb{Z}}(\mathcal{P} \cap \mathbb{Z}^p)$, coincides with the lattice $\mathbb{Z}^p$.
\end{itemize}
\end{dfn}

For the convenience of the reader, we briefly recall the notion of the integer affine span. Let $S \subseteq \mathbb{Z}^p$ be a non-empty set of lattice points. The \emph{integer affine span} of $S$, denoted by $\operatorname{aff}_{\mathbb{Z}}(S)$, is the set of all finite integer affine combinations of elements in $S$:
$$ \operatorname{aff}_{\mathbb{Z}}(S) = \left\{ \sum_{i=1}^k \lambda_i x_i \;\middle|\; k \in \mathbb{Z}_{\ge 1}, \, x_i \in S, \, \lambda_i \in \mathbb{Z}, \text{ and } \sum_{i=1}^k \lambda_i = 1 \right\}. $$

The following lemma is straightforward to verify. We will rely on this equivalent characterization in our subsequent proof of the spanning property.
\begin{lem}\label{lem-spanning-equ}
   $\operatorname{aff}_{\mathbb{Z}}(S) = \mathbb{Z}^p$ if and only if the origin $\mathbf{0} \in \mathbb{Z}^p$ can be expressed as an integer affine combination of points in $S$, and every standard basis vector $e_j$ of $\mathbb{Z}^p$ can be expressed as an integer linear combination of points in $S$ with coefficients summing to zero.
\end{lem}

We now give some definitions related to triangulations of polytopes.
For a detailed treatment on triangulations, we refer to~\cite{DeLoera2010}.

\begin{dfn}\label{Dfn-triangulation}
A \emph{triangulation} of a  $p$-polytope $\mathcal{P}$ is a finite collection $\mathcal{T}_\mathcal{P}$ of $p$-simplices with the following properties: 1.) $\mathcal{P} = \bigcup_{\Delta \in \mathcal{T}_\mathcal{P}} \Delta$. 2.) For every $\Delta_1, \Delta_2 \in \mathcal{T}_\mathcal{P}$, $\Delta_1 \cap \Delta_2$ is a face common to both $\Delta_1$ and $\Delta_2$.
\end{dfn}

\begin{dfn}
Let $\mathcal{P}$ be a lattice polytope. We define the following properties for a triangulation $\mathcal{T}$ of $\mathcal{P}$:
\begin{itemize}
\item $\mathcal{T}$ is said to be \emph{unimodular} if every simplex in $\mathcal{T}$ has a normalized volume of $1$.

\item $\mathcal{T}$ is said to be \emph{regular} if it is obtained by projecting the lower hull of a lifting of the vertices of $\mathcal{P}$.

\item $\mathcal{T}$ is termed \emph{flag} if its underlying simplicial complex is a flag complex, i.e., every minimal non-face has cardinality two.

\item $\mathcal{T}$ is called a \emph{quadratic triangulation} if it is simultaneously regular, unimodular, and flag.
\end{itemize}
\end{dfn}

For the unimodular triangulation, a definition is as follows. A lattice simplex $\Delta$ with vertices $v_0, \ldots, v_m$ is \emph{unimodular} if the vectors $v_m - v_0, v_{m-1} - v_0, \ldots, v_1 - v_0$ form a basis for the lattice $\mathrm{aff}(\Delta) \cap \mathbb{Z}^p$. A triangulation of a lattice polytope is a \emph{unimodular triangulation} if all its maximal dimensional simplices are unimodular. See~\cite[Definition 9.3.3]{DeLoera2010}.

\begin{thm}{\em \cite[Theorem 2.1]{Ferroni23}}
The following properties are increasing in strength:
\begin{enumerate}
    \item[\textup{(i)}] $\mathcal{P}$ is spanning.
    \item[\textup{(ii)}] $\mathcal{P}$ is very ample.
    \item[\textup{(iii)}] $\mathcal{P}$ has IDP.
    \item[\textup{(iv)}] $\mathcal{P}$ has a unimodular triangulation.
    \item[\textup{(v)}] $\mathcal{P}$ has a regular unimodular triangulation.
    \item[\textup{(vi)}] $\mathcal{P}$ has a quadratic triangulation.
\end{enumerate}
In other words, if $\mathcal{P}$ satisfies one of these properties, then it satisfies all of the previous ones.
\end{thm}

We refer to~\cite{Hofscheier}, \cite{Bruns2013}, \cite[Section 9]{Balletti2021}, and \cite{Ferroni23} for examples showing that
$$\text{Spanning}\centernot\Longrightarrow \text{Very ample} \centernot\Longrightarrow \text{IDP}.$$
We also refer to \cite{Haase2021} and \cite[Chapter 2.D]{Bruns-Gubeladze} for examples showing
$$\text{IDP}\centernot\Longrightarrow \text{Unimodular triangulation} \centernot\Longrightarrow \text{Regular unimodular triangulation}.$$

The following result shows that the Cartesian product of polytopes preserves the above properties.

\begin{prop}{\em \cite[Proposition 2.2]{Ferroni23}}\label{Proposition-Spanning-IDP-Very}
Let $\mathcal{P}$ and $\mathcal{Q}$ be lattice polytopes.
\begin{enumerate}
    \item[\textup{(i)}] If $\mathcal{P}$ and $\mathcal{Q}$ are spanning, then $\mathcal{P} \times \mathcal{Q}$ is spanning.
    \item[\textup{(ii)}] If $\mathcal{P}$ and $\mathcal{Q}$ are very ample, then $\mathcal{P} \times \mathcal{Q}$ is very ample.
    \item[\textup{(iii)}] If $\mathcal{P}$ and $\mathcal{Q}$ are IDP, then $\mathcal{P} \times \mathcal{Q}$ is IDP.
    \item[\textup{(iv)}] If $\mathcal{P}$ and $\mathcal{Q}$ have unimodular triangulations, then $\mathcal{P} \times \mathcal{Q}$ has a unimodular triangulation.
    \item[\textup{(v)}] If $\mathcal{P}$ and $\mathcal{Q}$ have regular unimodular triangulations, then $\mathcal{P} \times \mathcal{Q}$ has a regular unimodular triangulation.
    \item[\textup{(vi)}] If $\mathcal{P}$ and $\mathcal{Q}$ have quadratic triangulations, then $\mathcal{P} \times \mathcal{Q}$ has a quadratic triangulation.
\end{enumerate}
\end{prop}

\section{$h^*$-palindromic and $h^*$-unimodal}\label{Section-Four-unimod}

Let $f(x)=a_mx^m+a_{m-1}x^{m-1}+\cdots+a_1x+a_0$ be a polynomial with nonnegative real coefficients.
Such a polynomial is \emph{unimodal} if $a_0\leq \cdots \leq a_{i-1}\leq a_i\geq a_{i+1} \geq \cdots \geq a_m$ for some $0\leq i\leq m$;
\emph{log-concave} if $a_i^2\geq a_{i-1}a_{i+1}$ for $1\leq i\leq m-1$;
and \emph{real-rooted} if all its roots are real.
Real-rooted polynomials are log-concave, and log-concave polynomials with all positive coefficients are unimodal.
See \cite[Section 5]{RP.StanleyAC} or \cite{StanleyLog-concave}.

A $d$-dimensional polytope $\mathcal{P}$ is \emph{$h^*$-unimodal} if the coefficients of its $h^*$-polynomial satisfy $h_0\leq \cdots \leq h_{j-1}\leq h_j \geq h_{j+1}\geq \cdots \geq h_{d}$ for some $0\leq j\leq d$.
For background on $h^*$-unimodality, see \cite{Braum16}.

The polytope $\mathcal{P}$ is said to be \emph{Gorenstein of index $r$} \cite[Page 98]{BeckRobins} (for some positive integer $r$) if it satisfies the following lattice point conditions: $(r-1)\mathcal{P}^\circ \cap \mathbb{Z}^d = \emptyset$, $\left| r\mathcal{P}^\circ \cap \mathbb{Z}^d \right| = 1$, and $t\mathcal{P}^\circ \cap \mathbb{Z}^d = (t-r)\mathcal{P} \cap \mathbb{Z}^d$ for all integers $t > r$. For further background on Gorenstein polytopes, see \cite{Bruns-hvector}.
By \cite[Theorem 4.5]{BeckRobins}, a lattice polytope $\mathcal{P}$ is Gorenstein of index $r$ if and only if its $h^*$-polynomial is \emph{palindromic} of degree $p + 1 - r$, that is, it satisfies
$$ x^{p+1-r} h^*_{\mathcal{P}}(1/x) = h^*_{\mathcal{P}}(x). $$

\subsection{$h^*$-palindromic for the Cartesian product of polytopes}

Let $\mathcal{P}$ and $\mathcal{Q}$ be lattice polytopes. We highlight the following properties concerning the Cartesian product $\mathcal{P} \times \mathcal{Q}$:
\begin{enumerate}
    \item[\textup{(i)}] If $h^*_{\mathcal{P}}$ and $h^*_{\mathcal{Q}}$ are real-rooted polynomials, then $h^*_{\mathcal{P}\times \mathcal{Q}}$ is real rooted too. See \cite{Wagner1992} or \cite{Ferroni23}.
    \item[\textup{(ii)}] If $h^*_{\mathcal{P}}$ and $h^*_{\mathcal{Q}}$ are log-concave, is necessarily  $h^*_{\mathcal{P}\times \mathcal{Q}}$ log-concave too? This is an open question, see \cite{Ferroni23}.
    \item[\textup{(iii)}] If $h^*_{\mathcal{P}}$ and $h^*_{\mathcal{Q}}$ are unimodal, then $h^*_{\mathcal{P}\times \mathcal{Q}}$ is not necessarily unimodal. See \cite{Balletti}.
    \item[\textup{(iv)}] If $h^*_{\mathcal{P}}$ and $h^*_{\mathcal{Q}}$ are symmetric, then $h^*_{\mathcal{P}\times \mathcal{Q}}$ is not necessarily symmetric. See Example \ref{Exa-Gorenstein-Time}.
\end{enumerate}

\begin{exa}\label{Exa-Gorenstein-Time}
Let $\mathcal{P} = [0, 1]$ and $\mathcal{Q} = [0, 2]$ be $1$-polytopes in $\mathbb{R}$.
Both $\mathcal{P}$ and $\mathcal{Q}$ are Gorenstein. The polytope $\mathcal{P}$ has index $2$ with $h^*$-polynomial $h^*_{\mathcal{P}}(x) = 1$, while $\mathcal{Q}$ has index $1$ with $h^*_{\mathcal{Q}}(x) = 1 + x$.
Then $h^*_{\mathcal{P} \times \mathcal{Q}}(x) = 1 + 3x$, which is not symmetric.
\end{exa}

However, the Cartesian product of two Gorenstein polytopes of the same index remains Gorenstein.
We prove this using the following lemma, whose proof is immediate from the definition of a Gorenstein polytope.

\begin{lem}\label{Lem-Gorenstin-index-Ehrht}
Let $\mathcal{P} \subset \mathbb{R}^p$ be a lattice polytope, and let $\mathcal{P}^\circ$ denote its relative interior. Then $\mathcal{P}$ is a Gorenstein polytope with index $r$ if and only if
$$i(\mathcal{P}^\circ, t) = i(\mathcal{P}, t - r)\quad \text{for all}\quad t \in \mathbb{N},$$
where we take the convention $i(\mathcal{P}, m)=0$ if $m<0$; and $i(\mathcal{P}, m)=1$ if $m=0$.
\end{lem}

\begin{thm}\label{Theorem-Gorenstein-Cartesion}
Let $\mathcal{P} \subset \mathbb{R}^p$ and $\mathcal{Q} \subset \mathbb{R}^q$ be Gorenstein lattice polytopes of dimensions $p, q \ge 1$, with indices $r_{\mathcal{P}}$ and $r_{\mathcal{Q}}$, respectively. The Cartesian product $\mathcal{P} \times \mathcal{Q}$ is a Gorenstein polytope (i.e., $h^*_{\mathcal{P} \times \mathcal{Q}}(x)$ is palindromic) if and only if $r_{\mathcal{P}} = r_{\mathcal{Q}}$.
\end{thm}
\begin{proof}
We first prove the sufficiency.
Suppose that $r_{\mathcal{P}} = r_{\mathcal{Q}} = r$. Since both $\mathcal{P}$ and $\mathcal{Q}$ are Gorenstein, Lemma~\ref{Lem-Gorenstin-index-Ehrht} implies that $i(\mathcal{P}^\circ, t) = i(\mathcal{P}, t - r)$ and $i(\mathcal{Q}^\circ, t) = i(\mathcal{Q}, t - r)$.

For $t \ge r$, by Equations~\eqref{Equation-product-Hibi} and \eqref{Equation-product-circ}, we have
$$i((\mathcal{P} \times \mathcal{Q})^\circ, t)=i(\mathcal{P}^\circ, t)\cdot i(\mathcal{Q}^\circ, t) = i(\mathcal{P}, t - r)\cdot i(\mathcal{Q}, t - r) = i(\mathcal{P} \times \mathcal{Q}, t - r).$$

For $t < r$, the definition of the Cartesian product and the identity $(\mathcal{P} \times \mathcal{Q})^\circ = \mathcal{P}^\circ \times \mathcal{Q}^\circ$ imply that $t(\mathcal{P} \times \mathcal{Q})^\circ \cap \mathbb{Z}^{p+q+1} = \emptyset$. Thus, we obtain
$$i((\mathcal{P} \times \mathcal{Q})^\circ, t)=0 = i(\mathcal{P} \times \mathcal{Q}, t - r).$$
By Lemma~\ref{Lem-Gorenstin-index-Ehrht},  $\mathcal{P} \times \mathcal{Q}$ is Gorenstein with index $r$, and then $h^*_{\mathcal{P} \times \mathcal{Q}}(x)$ is palindromic.

Next, we prove the necessity. Assume $\mathcal{P} \times \mathcal{Q}$ is a Gorenstein polytope of index $r$.
Express the respective Ehrhart polynomials as
\begin{align*}
i(\mathcal{P}, t) &= a_p t^p + a_{p-1}t^{p-1} + \cdots+1, \quad a_p > 0;
\\ i(\mathcal{Q}, t) &= b_q t^q + b_{q-1}t^{q-1} + \cdots+1, \quad b_q > 0.
\end{align*}
By Lemma~\ref{Lem-Gorenstin-index-Ehrht}, we have
$i((\mathcal{P} \times \mathcal{Q})^\circ, t) = i(\mathcal{P} \times \mathcal{Q}, t - r).$
It follows that
\begin{align}\label{Equation-P-Q-index}
i(\mathcal{P}, t - r_{\mathcal{P}})\cdot i(\mathcal{Q}, t - r_{\mathcal{Q}}) = i(\mathcal{P}, t - r)\cdot i(\mathcal{Q}, t - r).
\end{align}
Let $\mathrm{LHS}$ and $\mathrm{RHS}$ denote the coefficients of $t^{p+q-1}$ on the left-hand and right-hand sides, respectively, of Equation \eqref{Equation-P-Q-index}.
We obtain
\begin{align*}
\mathrm{LHS} &= a_p b_{q-1} + b_q a_{p-1} - a_p b_q (p r_{\mathcal{P}} + q r_{\mathcal{Q}}),
\\ \mathrm{RHS} &= a_p b_{q-1} + b_q a_{p-1} - a_p b_q (p r + q r).
\end{align*}
Equating $\mathrm{LHS} = \mathrm{RHS}$ and dividing by $-a_p b_q \neq 0$ yields
$$p r_{\mathcal{P}} + q r_{\mathcal{Q}} = (p + q) r.$$

Geometrically, the index $r$ of a Gorenstein polytope $\mathcal{K}$ is the minimal positive integer dilation such that its relative interior contains a lattice point: $r = \min \{ t \in \mathbb{Z}_{>0} \mid i(\mathcal{K}^\circ, t) > 0 \}$. Since $(\mathcal{P} \times \mathcal{Q})^\circ$ contains a lattice point if and only if both $\mathcal{P}^\circ$ and $\mathcal{Q}^\circ$ contain lattice points simultaneously, we deduce that $r = \max(r_{\mathcal{P}}, r_{\mathcal{Q}})$.
Without loss of generality, assume $r_{\mathcal{P}} \le r_{\mathcal{Q}}$, which implies $r = r_{\mathcal{Q}}$. Substituting this into the derived linear relation
$$p r_{\mathcal{P}} + q r_{\mathcal{Q}} = p r_{\mathcal{Q}} + q r_{\mathcal{Q}} \implies p r_{\mathcal{P}} = p r_{\mathcal{Q}}.$$
It implies that $r_{\mathcal{P}} = r_{\mathcal{Q}}$. This completes the proof.
\end{proof}

\subsection{$h^*$-unimodal for the join of polytopes}

Let $\mathcal{P}\subseteq\mathbb{R}^p$ and $\mathcal{Q}\subseteq\mathbb{R}^q$ be lattice polytopes of dimensions $p$ and $q$, respectively.
By Proposition~\ref{Prop-Join-Properties}, the $h^*$-polynomial of $\mathcal{P}* \mathcal{Q}$ is given by
\begin{align*}
h^*_{\mathcal{P}*\mathcal{Q}}(x)= h^*_{\mathcal{P}}(x)\cdot  h^*_{\mathcal{Q}}(x).
\end{align*}

Clearly, the product of two real-rooted polynomials is again real-rooted.
By a simple derivation, one can obtain the product of two palindromic polynomials is again palindromic.
The product of two log-concave polynomials with nonnegative coefficients and no internal zeros is still log-concave, see \cite{Keilson-log}.
However, the product of two unimodal polynomials is not necessarily unimodal, see \cite[Proposition 3]{StanleyLog-concave}.

\begin{thm}\label{Theorem-Unimod-join}
If $h^*_{\mathcal{P}}$ and $h^*_{\mathcal{Q}}$ are unimodal, then $h^*_{\mathcal{P} *\mathcal{Q}}$ is not necessarily unimodal.
\end{thm}
\begin{proof}
Ferroni and Higashitani \cite[Theorem 5.5]{Ferroni23} constructed a polytope whose $h^*$-polynomial is unimodal but not log-concave.
This is a $5$-dimensional polytope $\mathcal{P}$ in $\mathbb{R}^7$ whose vertices are given by the columns of the following matrix:
\begin{align*}
\left(
\begin{array}{cccccccccccc}
1 & 1 & 1 & 1 & 0 & 0 & 0 & 0 & 0 & 0 & 0 & 0 \\
1 & 1 & 0 & 0 & 1 & 1 & 0 & 0 & 0 & 0 & 0 & 0 \\
0 & 0 & 0 & 0 & 1 & 1 & 1 & 1 & 0 & 0 & 0 & 0 \\
0 & 0 & 0 & 0 & 0 & 0 & 1 & 1 & 1 & 1 & 0 & 0 \\
0 & 0 & 0 & 0 & 0 & 0 & 0 & 0 & 1 & 1 & 1 & 1 \\
0 & 0 & 1 & 1 & 0 & 0 & 0 & 0 & 0 & 0 & 1 & 1 \\
111 & 112 & 0 & 1 & 0 & 1 & 0 & 1 & 0 & 1 & 0 & 1
\end{array}\right)
\end{align*}
The $h^*$-polynomial is $h^*_\mathcal{P}(x) = 1 + 6x + 6x^2 + 113x^3$, which is unimodal (but not log-concave).
The $h^*$-polynomial of $\mathcal{P}*\mathcal{P}$ is given by
$$h^*_{\mathcal{P} * \mathcal{P}}(x) = 1 + 12x + 48x^2 + 298x^3 + 1392x^4 + 1356x^5 + 12769x^6.$$
Obviously, this polynomial is not unimodal.
\end{proof}

\section{IDP, very ample, and spanning}\label{Section-Five-idp}

This section mainly studies the integer decomposition property, very ample property, and spanning property of the join of two lattice polytopes. Our main motivation for studying these objects comes from a result concerning the Cartesian product of two lattice polytopes, namely Proposition \ref{Proposition-Spanning-IDP-Very}.

\subsection{Integer decomposition property}

Before investigating the IDP property of the join, we present a basic lemma that will be employed frequently in subsequent arguments.

\begin{lem}\label{Lemma-IDP-two-convex-comb}
Let $\mathcal{P} \subset \mathbb{R}^p$ and $\mathcal{Q} \subset \mathbb{R}^q$ be two lattice polytopes.
Let $v$ be any point in $\mathcal{P} * \mathcal{Q}$. Then $v$ can be expressed as a convex combination of points from the embedded $\mathcal{P}$ and $\mathcal{Q}$:
$$v = \alpha(x, \mathbf{0}_q, 1) + (1-\alpha)(\mathbf{0}_p, y, 0) = (\alpha x, (1-\alpha)y, \alpha),$$
for some $x \in \mathcal{P}$, $y \in \mathcal{Q}$, and $\alpha \in [0, 1]$.

In particular, the set of lattice points of their join $\mathcal{P} * \mathcal{Q}$ is precisely given by
$$(\mathcal{P} * \mathcal{Q}) \cap \mathbb{Z}^{p+q+1} = \left( (\mathcal{P} \cap \mathbb{Z}^p) \times \{\mathbf{0}_q\} \times \{1\} \right) \cup \left( \{\mathbf{0}_p\} \times (\mathcal{Q} \cap \mathbb{Z}^q) \times \{0\} \right).$$
\end{lem}
\begin{proof}
Let the columns of matrices $X$ and $Y$ be the embedded vertices of $\mathcal{P}$ and $\mathcal{Q}$ in $\mathbb{R}^{p+q+1}$, respectively. Since $v \in \mathcal{P} * \mathcal{Q}$, there exist weight vectors $w_1, w_2 \ge \mathbf{0}$ such that $v = X w_1 + Y w_2$ and $\mathbf{1}^T w_1 + \mathbf{1}^T w_2 = 1$.

Let $\alpha = \mathbf{1}^T w_1$. Assuming $\alpha \in (0, 1)$ (the cases $\alpha \in \{0, 1\}$ are trivial), we can rewrite $v$ as
$$v = \alpha \left( X \frac{w_1}{\alpha} \right) + (1-\alpha) \left( Y \frac{w_2}{1-\alpha} \right).$$
The vector $X(w_1 / \alpha)$ is a convex combination of the embedded vertices of $\mathcal{P}$, yielding $X(w_1 / \alpha) = (x, \mathbf{0}_q, 1)$ for some $x \in \mathcal{P}$. Similarly, $Y(w_2 / (1-\alpha)) = (\mathbf{0}_p, y, 0)$ for some $y \in \mathcal{Q}$. Substituting these back into the equation for $v$ establishes the first claim.

For the second claim, the integrality of $v = (\alpha x, (1-\alpha)y, \alpha)$ trivially forces $\alpha \in \{0, 1\}$. This immediately restricts $v$ to either $(x, \mathbf{0}_q, 1)$ with $x \in \mathcal{P} \cap \mathbb{Z}^p$, or $(\mathbf{0}_p, y, 0)$ with $y \in \mathcal{Q} \cap \mathbb{Z}^q$, yielding the desired disjoint union.
\end{proof}

\begin{thm}\label{Theorem-IDP-Join}
Let $\mathcal{P} \subset \mathbb{R}^p$ and $\mathcal{Q} \subset \mathbb{R}^q$ be lattice polytopes possessing the integer decomposition property (IDP). Then their join $\mathcal{P} * \mathcal{Q}$ also possesses the integer decomposition property.
\end{thm}
\begin{proof}
Let $k \in \mathbb{Z}_{>0}$ and $z \in k(\mathcal{P} * \mathcal{Q}) \cap \mathbb{Z}^{p+q+1}$.
It suffices to show that $z$ is a sum of $k$ lattice points in $\mathcal{P} * \mathcal{Q}$.
By Lemma~\ref{Lemma-IDP-two-convex-comb}, the lattice point $z$ takes the form
$$z = k(\alpha x, (1-\alpha)y, \alpha) = (k\alpha x, k(1-\alpha)y, k\alpha), \quad \text{where } \alpha \in [0, 1].$$
Let $k_1 = k\alpha \in \mathbb{Z}$ and $k_2 = k(1-\alpha) = k - k_1$. Since $\alpha \in [0,1]$, it immediately follows that $k_1, k_2 \in \{0, 1, \dots, k\}$.

Writing $z = (k_1 x, k_2 y, k_1)$, we note that $k_1 x \in k_1 \mathcal{P} \cap \mathbb{Z}^p$ and $k_2 y \in k_2 \mathcal{Q} \cap \mathbb{Z}^q$. The IDP assumption on $\mathcal{P}$ and $\mathcal{Q}$ guarantees the existence of $u_i \in \mathcal{P} \cap \mathbb{Z}^p$ and $w_j \in \mathcal{Q} \cap \mathbb{Z}^q$ satisfying
$$ k_1 x = \sum_{i=1}^{k_1} u_i \quad \text{and} \quad k_2 y = \sum_{j=1}^{k_2} w_j.$$
Substituting these decompositions back into the expression for $z$, we obtain
$$z = \left( \sum_{i=1}^{k_1} u_i, \sum_{j=1}^{k_2} w_j, k_1 \right)=\sum_{i=1}^{k_1} (u_i, \mathbf{0}_q, 1) + \sum_{j=1}^{k_2} (\mathbf{0}_p, w_j, 0).$$
Observe that for each $i$, $(u_i, \mathbf{0}_q, 1) \in (\mathcal{P} * \mathcal{Q}) \cap \mathbb{Z}^{p+q+1}$, and for each $j$, $(\mathbf{0}_p, w_j, 0) \in (\mathcal{P} * \mathcal{Q}) \cap \mathbb{Z}^{p+q+1}$.
This decomposition uses $k_1 + k_2 = k$ points. The conclusion follows.
\end{proof}

\subsection{Very ample property}

\begin{thm}\label{Theorem-Veryample-Join}
The join of two very ample lattice polytopes is not necessarily very ample.
\end{thm}
\begin{proof}
We give a structural counterexample showing that if $\mathcal{P} \subset \mathbb{R}^p$ is a very ample polytope without the integer decomposition property (IDP), and $\mathcal{Q} \subset \mathbb{R}^q$ is any non-empty lattice polytope, then their join $\mathcal{P} * \mathcal{Q}$ is not very ample.
Such polytopes $\mathcal{P}$ exist by Section \ref{Section-2-Preliminary}.

By assumption, $\mathcal{P}$ does not possess the IDP. Hence there exists a positive integer $c$ and a lattice point $x \in c\mathcal{P} \cap \mathbb{Z}^p$ that cannot be written as a sum of $c$ lattice points in $\mathcal{P}$.
Since $\mathcal{P}$ is very ample, there exists a sufficiently large integer $k>c$ such that every lattice point in its $k$-th dilate decomposes into a sum of $k$ lattice points of the polytope.

Let $k_2 = k - c$. Since $\mathcal{Q}$ is non-empty, we may choose a lattice point $y \in k_2 \mathcal{Q} \cap \mathbb{Z}^q$  (with the convention that if $k_2 = 0$, $y = \mathbf{0}_q$).
Consider the point $z = (x, y, c) \in \mathbb{Z}^{p+q+1}$. We can rewrite $z$ as
\begin{align*}
z = k \left( \frac{c}{k} \cdot \frac{x}{c}, \frac{k-c}{k} \cdot \frac{y}{k-c}, \frac{c}{k} \right).
\end{align*}
Observe that $\frac{x}{c} \in \mathcal{P}$ and $\frac{y}{k-c} \in \mathcal{Q}$. It implies that $\frac{1}{k}z$ is a convex combination of
$$\left(\frac{x}{c}, \mathbf{0}_q, 1\right) \in \mathcal{P} \times \{\mathbf{0}_q\} \times \{1\}\quad \text{and}\quad
\left(\mathbf{0}_p, \frac{y}{k-c}, 0\right) \in \{\mathbf{0}_p\} \times \mathcal{Q} \times \{0\},$$
with coefficients $\frac{c}{k}$ and $\frac{k-c}{k}$, respectively. Hence, $z \in k(\mathcal{P} * \mathcal{Q}) \cap \mathbb{Z}^{p+q+1}$.

Assume, for the sake of contradiction, that $z$ admits an integer decomposition in $\mathcal{P} * \mathcal{Q}$.
That is, suppose $z$ can be written as the sum of $k$ lattice points belonging to $(\mathcal{P} * \mathcal{Q}) \cap \mathbb{Z}^{p+q+1}$. By Lemma \ref{Lemma-IDP-two-convex-comb}, the assumed decomposition takes the form
\begin{align}\label{equation-z-of-IDP} %tao: no since
z = \sum_{i=1}^{m_1} (u_i, \mathbf{0}_q, 1) + \sum_{j=1}^{m_2} (\mathbf{0}_p, v_j, 0),
\end{align}
where $m_1 + m_2 = k$, with $u_i \in \mathcal{P} \cap \mathbb{Z}^p$ and $v_j \in \mathcal{Q} \cap \mathbb{Z}^q$.

Comparing the final coordinate of both sides of Equation~\eqref{equation-z-of-IDP} yields $m_1 = c$.
Next, comparing the first $p$ coordinates yields
$$x = \sum_{i=1}^{c} u_i.$$
This shows that $x$ is a sum of exactly $c$ lattice points in $\mathcal{P}$, contradicting our initial assumption.

Consequently, $\mathcal{P} * \mathcal{Q}$ is not very ample. This completes the proof.
\end{proof}

\subsection{Spanning property}

\begin{thm}\label{Theorem-Spanning-Join}
Let $\mathcal{P} \subset \mathbb{R}^p$ and $\mathcal{Q} \subset \mathbb{R}^q$ be spanning lattice polytopes. Then their join $\mathcal{P} * \mathcal{Q}$ is also a spanning polytope.
\end{thm}
\begin{proof}
Let $M_\mathcal{P}$ and $M_\mathcal{Q}$ be the matrices whose columns are precisely the lattice points of $\mathcal{P}$ and $\mathcal{Q}$, respectively. Consider their augmented matrices
$$ \widehat{M}_\mathcal{P} = \begin{pmatrix} M_\mathcal{P} \\ \mathbf{1}^T \end{pmatrix} \quad \text{and} \quad \widehat{M}_\mathcal{Q} = \begin{pmatrix} M_\mathcal{Q} \\ \mathbf{1}^T \end{pmatrix}. $$
By Lemma~\ref{lem-spanning-equ}, there exist integer matrices $U$ and $V$ such that
$$ \widehat{M}_\mathcal{P} U = \begin{pmatrix} I_p & \mathbf{0} \\ \mathbf{0}^T & 1 \end{pmatrix} \quad \text{and} \quad \widehat{M}_\mathcal{Q} V = \begin{pmatrix} I_q & \mathbf{0} \\ \mathbf{0}^T & 1 \end{pmatrix}, $$
where $I_k$ denotes the $k \times k$ identity matrix.

By Lemma \ref{Lemma-IDP-two-convex-comb}, the augmented matrix of the join $\mathcal{P}*\mathcal{Q}$ is given by
$$ \widehat{M}_{\mathcal{P}*\mathcal{Q}} = \begin{pmatrix} M_\mathcal{P} & \mathbf{0}  \\ \mathbf{0} & M_\mathcal{Q} \\ \mathbf{1}^T & \mathbf{0}^T \\ \mathbf{1}^T & \mathbf{1}^T \end{pmatrix}. $$
Right-multiplication by the block diagonal matrix $\operatorname{diag}(U, V)$ gives
$$
\widehat{M}_{\mathcal{P}*\mathcal{Q}} \begin{pmatrix} U & \mathbf{0} \\ \mathbf{0} & V \end{pmatrix}
= \begin{pmatrix} M_\mathcal{P} U & \mathbf{0}\\ \mathbf{0} & M_\mathcal{Q} V \\ \mathbf{1}^T U & \mathbf{0}^T V \\ \mathbf{1}^T U & \mathbf{1}^T V \end{pmatrix}
= \begin{pmatrix}
I_p & \mathbf{0} & \mathbf{0} & \mathbf{0} \\
\mathbf{0} & \mathbf{0} & I_q & \mathbf{0} \\
\mathbf{0}^T & 1 & \mathbf{0}^T & 0 \\
\mathbf{0}^T & 1 & \mathbf{0}^T & 1
\end{pmatrix}.
$$
Finally, right-multiplying this result by the following matrix yields the identity matrix:
$$
\begin{pmatrix}
I_p & \mathbf{0} & \mathbf{0} & \mathbf{0} \\
\mathbf{0} & \mathbf{0} & I_q & \mathbf{0} \\
\mathbf{0}^T & 1 & \mathbf{0}^T & 0 \\
\mathbf{0}^T & 1 & \mathbf{0}^T & 1
\end{pmatrix}
\begin{pmatrix}
I_p & \mathbf{0} & \mathbf{0} & \mathbf{0} \\
\mathbf{0}^T & \mathbf{0}^T & 1 & 0 \\
\mathbf{0} & I_q & \mathbf{0} & \mathbf{0} \\
\mathbf{0}^T & \mathbf{0}^T & -1 & 1
\end{pmatrix}
= I_{p+q+2}.
$$
This demonstrates that the augmented matrix $\widehat{M}_{\mathcal{P}*\mathcal{Q}}$ can be transformed into the identity matrix via integer column operations, confirming that its integer column span is the entire ambient space. Thus, $\mathcal{P} * \mathcal{Q}$ is a spanning polytope.
\end{proof}

\section{Triangulation}\label{Section-Six-Trign}

This section focuses on triangulations of the join of two lattice polytopes, in particular unimodular triangulations, regular triangulations, and quadratic triangulations.
Before that, we state the following preparatory lemma.

\subsection{Preliminary lemma}

A result analogous to Lemma \ref{Lemma-triangulation-join} can be found in \cite[Theorem 4.2.7]{DeLoera2010}, concerning the join of any two subdivisions of point configurations.
Up to affine unimodular equivalence, we assume throughout this section that all lattice polytopes are full-dimensional.

\begin{lem}\label{Lemma-triangulation-join}
Let $\mathcal{P} \subset \mathbb{R}^p$ and $\mathcal{Q} \subset \mathbb{R}^q$ be polytopes with triangulations $\mathcal{T}_\mathcal{P}$ and $\mathcal{T}_\mathcal{Q}$, respectively. The collection of simplices defined by
$$\mathcal{T}_{\mathcal{P}*\mathcal{Q}} = \{ \Delta_\mathcal{P} * \Delta_\mathcal{Q} \mid \Delta_\mathcal{P} \in \mathcal{T}_\mathcal{P}, \Delta_\mathcal{Q} \in \mathcal{T}_\mathcal{Q} \}$$
forms a triangulation of the join polytope $\mathcal{P} * \mathcal{Q}$.
\end{lem}

\begin{proof}
For any $x \in \mathcal{P}$ and $y \in \mathcal{Q}$, let $\widetilde{x} = (x, \mathbf{0}_q, 1)$ and $\widetilde{y} = (\mathbf{0}_p, y, 0)$. Similarly, for simplices $\Delta_\mathcal{P} \in \mathcal{T}_\mathcal{P}$ and ${\Delta}_\mathcal{Q} \in \mathcal{T}_\mathcal{Q}$, we define their embeddings as $\widetilde{\Delta}_\mathcal{P} = \Delta_\mathcal{P} \times \{\mathbf{0}_q\} \times \{1\}$ and $\widetilde{\Delta}_\mathcal{Q} = \{\mathbf{0}_p\} \times \Delta_\mathcal{Q} \times \{0\}$. By Proposition~\ref{Prop-Join-Properties}, the join of any two full-dimensional simplices is again a full-dimensional simplex.

First, we verify that the union of all simplices in $\mathcal{T}_{\mathcal{P}*\mathcal{Q}}$ equals $\mathcal{P} * \mathcal{Q}$. Let $z \in \mathcal{P} * \mathcal{Q}$. By Lemma~\ref{Lemma-IDP-two-convex-comb}, there exist $x \in \mathcal{P}$, $y \in \mathcal{Q}$, and $\alpha \in [0, 1]$ such that
$$ z = \alpha \widetilde{x} + (1 - \alpha) \widetilde{y}. $$
Because $\mathcal{T}_\mathcal{P}$ and $\mathcal{T}_\mathcal{Q}$ are triangulations of $\mathcal{P}$ and $\mathcal{Q}$, respectively, there exist simplices $\Delta_\mathcal{P} \in \mathcal{T}_\mathcal{P}$ and $\Delta_\mathcal{Q} \in \mathcal{T}_\mathcal{Q}$ such that $x \in \Delta_\mathcal{P}$ and $y \in \Delta_\mathcal{Q}$. Consequently, $z \in \Delta_\mathcal{P} * \Delta_\mathcal{Q}$, which implies that $\mathcal{P} * \mathcal{Q} \subseteq \bigcup_{\sigma \in \mathcal{T}_{\mathcal{P}*\mathcal{Q}}} \sigma$.
The reverse inclusion, $\bigcup_{\sigma \in \mathcal{T}_{\mathcal{P}*\mathcal{Q}}} \sigma \subseteq \mathcal{P} * \mathcal{Q}$, follows trivially from the definition of the join. This establishes the set equality $\bigcup_{\sigma \in \mathcal{T}_{\mathcal{P}*\mathcal{Q}}} \sigma = \mathcal{P} * \mathcal{Q}$.

Next, we must demonstrate that the intersection of any two simplices in $\mathcal{T}_{\mathcal{P}*\mathcal{Q}}$ is a common face of both. Let $\sigma = \Delta_{\mathcal{P},1} * \Delta_{\mathcal{Q},1}$ and $\tau = \Delta_{\mathcal{P},2} * \Delta_{\mathcal{Q},2}$ be two arbitrary simplices in $\mathcal{T}_{\mathcal{P}*\mathcal{Q}}$. Since $\mathcal{T}_\mathcal{P}$ and $\mathcal{T}_\mathcal{Q}$ are triangulations, the intersections $F_\mathcal{P} = \Delta_{\mathcal{P},1} \cap \Delta_{\mathcal{P},2}$ and $F_\mathcal{Q} = \Delta_{\mathcal{Q},1} \cap \Delta_{\mathcal{Q},2}$ are common faces of the respective simplices.

Applying Lemma~\ref{Lemma-IDP-two-convex-comb}, we obtain
$$
\sigma \cap \tau = (\Delta_{\mathcal{P},1} * \Delta_{\mathcal{Q},1}) \cap (\Delta_{\mathcal{P},2} * \Delta_{\mathcal{Q},2}) = (\Delta_{\mathcal{P},1} \cap \Delta_{\mathcal{P},2}) * (\Delta_{\mathcal{Q},1} \cap \Delta_{\mathcal{Q},2})= F_\mathcal{P} * F_\mathcal{Q}.
$$
By the supporting hyperplane characterization of faces, let $u_1,u_2$ be the normal vectors exposing $F_{\mathcal{P}},F_{\mathcal{Q}}$, and let $m_1,m_2$ be the corresponding maximum inner products over $\Delta_{\mathcal{P},1},\Delta_{\mathcal{Q},1}$.

Define $U = (u_1, u_2, m_2 - m_1)^T \in \mathbb{R}^{p+q+1}$. Any $z \in \sigma = \Delta_{\mathcal{P},1} * \Delta_{\mathcal{Q},1}$ takes the form $z = (\alpha x, (1-\alpha)y, \alpha)$ for $x \in \Delta_{\mathcal{P},1}, y \in \Delta_{\mathcal{Q},1}$, and $\alpha \in [0,1]$. Evaluating the inner product yields:
$$ \langle U, z \rangle = \alpha \langle u_1, x \rangle + (1-\alpha)\langle u_2, y \rangle + \alpha(m_2 - m_1) \le \alpha m_1 + (1-\alpha)m_2 + \alpha(m_2 - m_1) = m_2. $$

Equality holds if and only if $\alpha(m_1 - \langle u_1, x \rangle) + (1-\alpha)(m_2 - \langle u_2, y \rangle) = 0$. Since both summands are non-negative, this is equivalent to $x \in F_{\mathcal{P}}$ (for $\alpha > 0$) and $y \in F_{\mathcal{Q}}$ (for $\alpha < 1$), which precisely parametrizes the point set of $F_{\mathcal{P}} * F_{\mathcal{Q}}$.

Hence, the vector $U$ exposes $F_{\mathcal{P}} * F_{\mathcal{Q}}$ as a face of $\sigma$. By identical reasoning on $\tau = \Delta_{\mathcal{P},2} * \Delta_{\mathcal{Q},2}$ (since $F_{\mathcal{P}}$ and $F_{\mathcal{Q}}$ are also faces of $\Delta_{\mathcal{P},2}$ and $\Delta_{\mathcal{Q},2}$), there exists a corresponding normal vector $U'$ that exposes $F_{\mathcal{P}} * F_{\mathcal{Q}}$ as a face of $\tau$.

Therefore the join $F_\mathcal{P} * F_\mathcal{Q}$ is a common face of both $\sigma$ and $\tau$. Consequently, $\mathcal{T}_{\mathcal{P}*\mathcal{Q}}$ forms a triangulation of the join polytope $\mathcal{P} * \mathcal{Q}$.
\end{proof}

\subsection{Unimodular triangulation}

\begin{thm}\label{Theorem-Unimodular-Triangulation}
If $\mathcal{P}$ and $\mathcal{Q}$ admit unimodular triangulations $\mathcal{T}_\mathcal{P}$ and $\mathcal{T}_\mathcal{Q}$, respectively, then their join $\mathcal{P} * \mathcal{Q} \subset \mathbb{R}^{p+q+1}$ also admits a unimodular triangulation \(\mathcal{T}_{\mathcal{P}*\mathcal{Q}}\).
\end{thm}

\begin{proof}
To prove that the induced triangulation $\mathcal{T}_{\mathcal{P}*\mathcal{Q}}$ is unimodular, it suffices to show that $\Delta_{\mathcal{P}*\mathcal{Q}} = \Delta_\mathcal{P} * \Delta_\mathcal{Q} \in \mathcal{T}_{\mathcal{P}*\mathcal{Q}}$ is a unimodular simplex, where $\Delta_\mathcal{P} \in \mathcal{T}_\mathcal{P}$ and $\Delta_\mathcal{Q} \in \mathcal{T}_\mathcal{Q}$. Let $\widehat{V}_{1}$ and $\widehat{V}_{2}$ denote their respective augmented vertex matrices (formed by appending a bottom row of $1$s) of $\Delta_\mathcal{P}$ and $\Delta_\mathcal{Q}$.

We recall a fundamental equivalent characterization of a  full-dimensional unimodular simplex:  A lattice simplex is unimodular if and only if the absolute value of the determinant of its augmented vertex matrix is exactly $1$. Thus, we have
\begin{equation}\label{det-v1-v2}
\left| \det \widehat{V}_{1} \right| = 1 \quad \text{and} \quad
\left| \det \widehat{V}_{2} \right| = 1.
\end{equation}

Let $\widehat{M}$ be the augmented vertex matrix of $\Delta_\mathcal{P}* \Delta_\mathcal{Q}$.   Then by Lemma \ref{Lemma-IDP-two-convex-comb}, we have
$$ \widehat{M} = \begin{pmatrix} V_1 & \mathbf{0} \\ \mathbf{0} & V_2 \\ \mathbf{1}^T & \mathbf{0}^T \\ \mathbf{1}^T & \mathbf{1}^T \end{pmatrix},$$
where \(V_1\) and \(V_2\) are the vertex matrices of \(\Delta_\mathcal{P}\) and \(\Delta_\mathcal{Q}\), respectively.
Clearly, the absolute value of the determinant of \(\widehat{M}\) equals that of the matrix below.
\[
\widetilde{M}   = \begin{pmatrix} V_1 & \mathbf{0} \\ \mathbf{1}^T & \mathbf{0}^T \\ \mathbf{0} & V_2 \\ \mathbf{0}^T & \mathbf{1}^T \end{pmatrix}
= \begin{pmatrix}
\widehat{V}_{1}& \mathbf{0} \\
\mathbf{0} & \widehat{V}_{2}
\end{pmatrix}.
\]
By Equation~\eqref{det-v1-v2}, we obtain
$$|\det \widehat{M}| = |\det \widetilde{M} | = |\det \widehat{V}_{1}| \cdot |\det \widehat{V}_{2}|=1.$$
This proves that the simplex $\Delta_{\mathcal{P}*\mathcal{Q}}$ is unimodular. Thus, $\mathcal{T}_{\mathcal{P}*\mathcal{Q}}$ is a unimodular triangulation of $\mathcal{P} * \mathcal{Q}$. This completes the proof.
\end{proof}

\subsection{Regular unimodular triangulation}

Given a finite point configuration $A = \{a_1, a_2, \ldots, a_n\} \subset \mathbb{R}^p$ and the polytope $\mathcal{P} = \mathrm{conv}(A)$, randomly choose heights $h_1, h_2, \ldots, h_n \in \mathbb{R}$, and define the lifted polytope
$$\widehat{\mathcal{P}} := \mathrm{conv}\left\{(a_1, h_1), (a_2, h_2), \ldots, (a_n, h_n)\right\} \subseteq \mathbb{R}^{p+1}.$$

Following \cite[Chapter 3]{BeckRobins}, the \emph{lower hull} of $\widehat{\mathcal{P}}$ is the set of points $(x_1, \dots, x_{p+1}) \in \widehat{\mathcal{P}}$ visible from below (i.e., $(x_1, \dots, x_{p+1} - \epsilon) \notin \widehat{\mathcal{P}}$ for all $\epsilon > 0$). A face contained within this hull is called a \emph{lower face}. Equivalently, as noted in \cite[p.~55]{DeLoera2010}, a lower face is formed by the points of $\widehat{\mathcal{P}}$ satisfying a non-vertical hyperplane equation $x_{p+1} = H(x_1, \dots, x_p)$, subject to the supporting condition $H(y_1, \dots, y_p) \le y_{p+1}$ for all points $(y_1, \dots, y_{p+1}) \in \widehat{\mathcal{P}}$.

\begin{exa}
Let $\mathcal{P} = \operatorname{conv}\{v_1=(0,0),\, v_2=(1,0),\, v_3=(1,1),\, v_4=(0,1)\}$. To explicitly construct a regular triangulation of $\mathcal{P}$, we introduce a height function $\omega$ that lifts the vertices to $\widehat{v}_i = (v_i, \omega(v_i))$, forming the lifted polytope
$$ \widehat{\mathcal{P}} = \operatorname{conv}\{\widehat{v}_1=(0,0,0),\, \widehat{v}_2=(1,0,1),\, \widehat{v}_3=(1,1,0),\, \widehat{v}_4=(0,1,1)\} \subset \mathbb{R}^3. $$
The regular triangulation induced by $\omega$ is precisely the vertical projection of the lower hull of $\widehat{\mathcal{P}}$ onto $\mathbb{R}^2$.

The plane $z = x - y$ supports $\widehat{\mathcal{P}}$ from below. The points achieving equality are exactly $\{ \widehat{v}_1, \widehat{v}_2, \widehat{v}_3 \}$.
Similarly, the plane $z = y - x$ is another lower supporting hyperplane.

Consequently, the lower hull of $\widehat{\mathcal{P}}$ consists precisely of the two $2$-dimensional simplices $\operatorname{conv}(\widehat{v}_1, \widehat{v}_2, \widehat{v}_3)$ and $\operatorname{conv}(\widehat{v}_1, \widehat{v}_3, \widehat{v}_4)$.

Applying the canonical projection $\pi(x, y, z) = (x, y)$ to these lower faces, we obtain the $2$-simplices $\Delta_1 = \operatorname{conv}(v_1, v_2, v_3)$ and $\Delta_2 = \operatorname{conv}(v_1, v_3, v_4)$. These simplices intersect precisely at their common facet, the diagonal $\operatorname{conv}(v_1, v_3)$. By definition, the collection $\mathcal{T} = \{\Delta_1, \Delta_2\}$ forms a regular triangulation of $\mathcal{P}$.
\end{exa}

\begin{thm}\label{Theorem-Regular-Uniom-Triangu}
 If $\mathcal{P}$ and $\mathcal{Q}$ both admit regular unimodular triangulations $\mathcal{T}_\mathcal{P}$ and $\mathcal{T}_\mathcal{Q}$, then their join $\mathcal{P} * \mathcal{Q} \subset \mathbb{R}^{p+q+1}$ also admits a regular unimodular triangulation \(\mathcal{T}_{\mathcal{P}*\mathcal{Q}}\).
\end{thm}

\begin{proof}
Let $\mathcal{T}_\mathcal{P}$ and $\mathcal{T}_\mathcal{Q}$ be regular unimodular triangulations of $\mathcal{P}$ and $\mathcal{Q}$, respectively.

By Theorem~\ref{Theorem-Unimodular-Triangulation}, we know that the combinatorial join $\mathcal{T}_{\mathcal{P}*\mathcal{Q}} = \{ \Delta_\mathcal{P} * \Delta_\mathcal{Q} \mid \Delta_\mathcal{P} \in \mathcal{T}_\mathcal{P}, \Delta_\mathcal{Q} \in \mathcal{T}_\mathcal{Q} \}$ is a unimodular triangulation of $\mathcal{P} * \mathcal{Q}$. Now we will construct a joint height function $\omega$ to prove that $\mathcal{T}_{\mathcal{P}*\mathcal{Q}}$ is exactly the regular triangulation induced by $\omega$.

By definition, the regular triangulation $\mathcal{T}_\mathcal{P}$ arises from a height function $\omega_\mathcal{P}: \mathcal{P} \cap \mathbb{Z}^p \to \mathbb{R}$ as the projection of the lower hull of
$$ \widehat{\mathcal{P}} = \operatorname{conv}\{ (u, \omega_\mathcal{P}(u)) \mid u \in \mathcal{P} \cap \mathbb{Z}^p \} \subset \mathbb{R}^{p+1}. $$
Likewise, $\mathcal{T}_\mathcal{Q}$ is induced by a height function $\omega_\mathcal{Q}: \mathcal{Q} \cap \mathbb{Z}^q \to \mathbb{R}$ via the lower hull projection of
$$ \widehat{\mathcal{Q}} = \operatorname{conv}\{ (v, \omega_\mathcal{Q}(v)) \mid v \in \mathcal{Q} \cap \mathbb{Z}^q \} \subset \mathbb{R}^{q+1}. $$

We define the joint height function $\omega: (\mathcal{P} * \mathcal{Q}) \cap \mathbb{Z}^{p+q+1} \to \mathbb{R}$.
By Lemma~\ref{Lemma-IDP-two-convex-comb}, the lattice points of the join $\mathcal{P} * \mathcal{Q}$ consist exactly of the canonical embeddings $\widetilde{u} = (u, \mathbf{0}_q, 1)$ for $u \in \mathcal{P} \cap \mathbb{Z}^p$ and $\widetilde{v} = (\mathbf{0}_p, v, 0)$ for $v \in \mathcal{Q} \cap \mathbb{Z}^q$. We set
$$ \omega(\widetilde{u}) = \omega_\mathcal{P}(u), \quad \text{and} \quad \omega(\widetilde{v}) = \omega_\mathcal{Q}(v). $$

To prove that $\mathcal{T}_{\mathcal{P}*\mathcal{Q}}$ is the regular triangulation induced by $\omega$, we must show that for any unimodular simplex $\Delta = \Delta_\mathcal{P} * \Delta_\mathcal{Q} \in \mathcal{T}_{\mathcal{P}*\mathcal{Q}}$, its lifted vertices form a lower facet of the lifted polytope $\widehat{\mathcal{P} * \mathcal{Q}}$. This is equivalent to finding a global affine functional $H(x, y, t) = \langle c_\mathcal{P}, x \rangle + \langle c_\mathcal{Q}, y \rangle + c_t t + c_0$ that strictly supports the lifted simplex $\widehat{\Delta}$.

Since $\Delta_\mathcal{P} \in \mathcal{T}_\mathcal{P}$, there exists a strict lower supporting functional $h_\mathcal{P}(u) = \langle c_\mathcal{P}, u \rangle + a_\mathcal{P}$ for $\widehat{\mathcal{P}}$. That is, $h_\mathcal{P}(u) \le \omega_\mathcal{P}(u)$ for all $u \in \mathcal{P} \cap \mathbb{Z}^p$, with equality holding if and only if $u \in \operatorname{vert}(\Delta_\mathcal{P})$.
Similarly, since $\Delta_\mathcal{Q} \in \mathcal{T}_\mathcal{Q}$, there exists a strict lower supporting functional $h_\mathcal{Q}(v) = \langle c_\mathcal{Q}, v \rangle + a_\mathcal{Q}$ for $\widehat{\mathcal{Q}}$, with equality holding if and only if $v \in \operatorname{vert}(\Delta_\mathcal{Q})$.

We construct the global functional $H$ by setting $c_0 = a_\mathcal{Q}$ and $c_t = a_\mathcal{P} - a_\mathcal{Q}$. Let us evaluate $H$ on the lattice points of $\mathcal{P} * \mathcal{Q}$:
For any $\widetilde{u} = (u, \mathbf{0}_q, 1)$, we have:
$$ H(\widetilde{u}) = \langle c_\mathcal{P}, u \rangle + c_t \cdot 1 + c_0 = \langle c_\mathcal{P}, u \rangle + a_\mathcal{P} = h_\mathcal{P}(u) \le \omega_\mathcal{P}(u) = \omega(\widetilde{u}). $$
For any $\widetilde{v} = (\mathbf{0}_p, v, 0)$, we have:
$$ H(\widetilde{v}) = \langle c_\mathcal{Q}, v \rangle + c_t \cdot 0 + c_0 = \langle c_\mathcal{Q}, v \rangle + a_\mathcal{Q} = h_\mathcal{Q}(v) \le \omega_\mathcal{Q}(v) = \omega(\widetilde{v}). $$

By the strictness of $h_\mathcal{P}$ and $h_\mathcal{Q}$, the equality $H(\widetilde{z}) = \omega(\widetilde{z})$ holds if and only if $\widetilde{z}$ is an embedding of a vertex from $\Delta_\mathcal{P}$ or $\Delta_\mathcal{Q}$. Thus, $H$ is a strict lower supporting hyperplane for the lifted simplex $\widehat{\Delta_\mathcal{P} * \Delta_\mathcal{Q}}$.

By the construction above, each simplex in $\mathcal{T}_{\mathcal{P}*\mathcal{Q}}$ corresponds bijectively to a distinct lower facet of $\widehat{\mathcal{P} * \mathcal{Q}}$. Because the simplices of $\mathcal{T}_{\mathcal{P}*\mathcal{Q}}$ already form a complete geometric partition of the base polytope $\mathcal{P} * \mathcal{Q}$, their corresponding lower facets have completely exhausted the lower hull of $\widehat{\mathcal{P} * \mathcal{Q}}$ (any additional lower facet would force an overlapping projection, violating the convexity of the lower hull). Therefore, the regular triangulation induced by $\omega$ is precisely $\mathcal{T}_{\mathcal{P}*\mathcal{Q}}$. Therefore, $\mathcal{T}_{\mathcal{P}*\mathcal{Q}}$ is a regular unimodular triangulation.

\end{proof}

\subsection{Quadratic triangulation}

Let $V = \{x_1, \ldots, x_n\}$ be a finite set, called a \emph{vertex set}. An \emph{(abstract) simplicial complex} on $V$ is a collection
$\Gamma$ of subsets $F$ of $V$ satisfying:
\begin{enumerate}
\item[\textup{(i)}] $\{x_i\} \in \Gamma$ for $1 \leq i \leq n$,
\item[\textup{(ii)}] If $F \in \Gamma$ and $G \subseteq F$, then $G \in \Gamma$.
\end{enumerate}
An element $F$ of $\Gamma$ is called a \emph{face}. A maximal face $F$, i.e., a face that is not contained in any larger face, is called a \emph{facet}. The \emph{dimension} of $F$ is $\dim F=\#F - 1$.
An $i$-dimensional face is called an \emph{$i$-face}.
The \emph{dimension} of $\Gamma$ is defined to be the maximum dimension of a face of $\Gamma$.

For a detailed treatment of simplicial complexes, we refer to \cite{RP.StanleyCO-Al} and \cite[Chapter 12]{RP.StanleyAC}.

General simplicial complexes for which every minimal non-face has two elements are called \emph{flag complexes}.
Let $\Gamma$ be a $(d-1)$-dimensional simplicial complex. For $0\leq k\leq d-1$, the \emph{$k$-skeleton} $\Gamma_k$ of $\Gamma$ is defined by $\Gamma_k=\{F\in \Gamma \mid \dim F\leq k\}$.

Let $G = (V, E)$ be a finite, simple, undirected graph, where $V$ is the vertex set and $E$ is the edge set. A subset of vertices $C \subseteq V$ is called a \emph{clique }of $G$ if every pair of distinct vertices $u, v \in C$ is mutually adjacent, i.e., $\{u, v\} \in E$. Equivalently, $C$ is a clique if and only if the induced subgraph $G[C]$ is a complete graph.

The \emph{clique complex} (sometimes referred to as the flag complex of the graph) of $G$, denoted by $\Gamma(G)$, is the abstract simplicial complex on the vertex set $V$ whose faces are precisely the cliques of $G$. Formally, a subset $\sigma \subseteq V$ is a face of $\Gamma(G)$ if and only if $\sigma$ is a clique in $G$.

\begin{rem}
According to Definition~\ref{Dfn-triangulation}, a triangulation $\mathcal{T}_\mathcal{P}$ is defined geometrically as a collection of maximal $p$-simplices. However, to seamlessly apply tools from topological combinatorics, we standardly identify $\mathcal{T}_\mathcal{P}$ with its associated downward-closed abstract simplicial complex (i.e., the complex containing all these maximal simplices alongside all of their faces).

By a mild abuse of notation, we use $\mathcal{T}_\mathcal{P}$ to denote both the geometric triangulation and this abstract complex. Therefore, when we write $\sigma \in \mathcal{T}_\mathcal{P}$, it formally implies that $\sigma$ is a valid face of some maximal simplex in the triangulation.
\end{rem}

\begin{lem}{\em \cite[Chapter 9.1]{Kozlov}}\label{Lemma-Flag-Clique}
Let $\Gamma$ be an abstract simplicial complex, and let $G$ denote its $1$-skeleton (the underlying graph formed by its vertices and $1$-dimensional simplices). The complex $\Gamma$ is a flag complex if and only if every clique in $G$ constitutes a valid face in $\Gamma$. Equivalently, $\Gamma$ is a flag complex if and only if it coincides with the clique complex of its $1$-skeleton, i.e., $\Gamma = \Gamma(G)$.
\end{lem}

\begin{thm}\label{Theorem-Quadratic-Triang-Join}
If $\mathcal{P}$ and $\mathcal{Q}$ both admit quadratic triangulations $\mathcal{T}_\mathcal{P}$ and $\mathcal{T}_\mathcal{Q}$, then their join $\mathcal{P} * \mathcal{Q}$ also admits a quadratic triangulation $\mathcal{T}_{\mathcal{P}*\mathcal{Q}}$.
\end{thm}
\begin{proof}
By definition, a triangulation is quadratic if and only if it is unimodular, regular, and flag.
By Theorem \ref{Theorem-Regular-Uniom-Triangu}, it remains to demonstrate that $\mathcal{T}_{\mathcal{P}*\mathcal{Q}}$ is a flag complex.

By Proposition~\ref{Prop-Join-Properties}, we have
$$V = \operatorname{vert}(\mathcal{T}_{\mathcal{P}*\mathcal{Q}}) = \rho_\mathcal{P}(\operatorname{vert}(\mathcal{T}_\mathcal{P})) \cup \rho_\mathcal{Q}(\operatorname{vert}(\mathcal{T}_\mathcal{Q})),$$
where $\rho_\mathcal{P}$ and $\rho_\mathcal{Q}$ are the canonical embeddings into $\mathbb{R}^{p+q+1}$.

By Lemma \ref{Lemma-Flag-Clique}, a simplicial complex $\mathcal{T}$ is a flag complex (or a clique complex) if every clique (complete subgraph) in its $1$-skeleton corresponds to the vertex set of a face in $\mathcal{T}$.
Let $\mathcal{G}_{\mathcal{P}*\mathcal{Q}}$ denote the $1$-skeleton (the underlying graph) of $\mathcal{T}_{\mathcal{P}*\mathcal{Q}}$.
By the definition of the join operation for simplicial complexes, the edge set $E$ of $\mathcal{G}_{\mathcal{P}*\mathcal{Q}}$ is constructed by taking the edges of $\mathcal{T}_\mathcal{P}$, the edges of $\mathcal{T}_\mathcal{Q}$, and all possible bipartite cross-edges connecting any vertex in $\mathcal{T}_\mathcal{P}$ to any vertex in $\mathcal{T}_\mathcal{Q}$ (i.e., for any $u \in \operatorname{vert}(\mathcal{T}_\mathcal{P})$ and $v \in \operatorname{vert}(\mathcal{T}_\mathcal{Q})$, the edge $\{\rho_\mathcal{P}(u), \rho_\mathcal{Q}(v)\} \in E$).

To prove the flag property, let $C \subseteq V$ be an arbitrary clique in $\mathcal{G}_{\mathcal{P}*\mathcal{Q}}$. We must show that $C$ is the vertex set of a simplex in $\mathcal{T}_{\mathcal{P}*\mathcal{Q}}$. We define the corresponding vertex sets in $\mathcal{T}_\mathcal{P}$ and $\mathcal{T}_\mathcal{Q}$ via the pre-images of $C$ under the canonical embeddings:
$$C_\mathcal{P} = \{ u \in \operatorname{vert}(\mathcal{T}_\mathcal{P}) \mid \rho_\mathcal{P}(u) \in C \}, \quad \text{and} \quad C_\mathcal{Q} = \{ v \in \operatorname{vert}(\mathcal{T}_\mathcal{Q}) \mid \rho_\mathcal{Q}(v) \in C \}.$$

Because $C$ is a clique in $\mathcal{G}_{\mathcal{P}*\mathcal{Q}}$, any two distinct vertices within $\rho_\mathcal{P}(C_\mathcal{P})$ must be connected by an edge in $E$. Since cross-edges only exist between $\mathcal{P}$ and $\mathcal{Q}$, this implies that any two vertices in $C_\mathcal{P}$ are connected by an edge in the $1$-skeleton of $\mathcal{T}_\mathcal{P}$. Therefore, $C_\mathcal{P}$ forms a clique in the $1$-skeleton of $\mathcal{T}_\mathcal{P}$.

By the hypothesis that $\mathcal{T}_\mathcal{P}$ is a quadratic triangulation, it is inherently flag. Consequently, the clique $C_\mathcal{P}$ spans a valid face in $\mathcal{T}_\mathcal{P}$. Let this face be denoted by $\sigma_\mathcal{P} \in \mathcal{T}_\mathcal{P}$ (where $\sigma_\mathcal{P} = \emptyset$ if $C_\mathcal{P} = \emptyset$).

Similarly, $C_\mathcal{Q}$ forms a clique in the $1$-skeleton of $\mathcal{T}_\mathcal{Q}$. Since $\mathcal{T}_\mathcal{Q}$ is flag, $C_\mathcal{Q}$ spans a valid face $\sigma_\mathcal{Q} \in \mathcal{T}_\mathcal{Q}$.

By Lemma~\ref{Lemma-triangulation-join}, if $\sigma_\mathcal{P} \in \mathcal{T}_\mathcal{P}$ and $\sigma_\mathcal{Q} \in \mathcal{T}_\mathcal{Q}$ are faces, their join $\sigma_\mathcal{P} * \sigma_\mathcal{Q}$ is a face in $\mathcal{T}_{\mathcal{P}*\mathcal{Q}}$. The vertex set of this join simplex is exactly
$$\operatorname{vert}(\sigma_\mathcal{P} * \sigma_\mathcal{Q}) = \rho_\mathcal{P}(\operatorname{vert}(\sigma_\mathcal{P})) \cup \rho_\mathcal{Q}(\operatorname{vert}(\sigma_\mathcal{Q})) = \rho_\mathcal{P}(C_\mathcal{P}) \cup \rho_\mathcal{Q}(C_\mathcal{Q}) = C.$$
Thus, the arbitrary clique $C$ precisely spans a simplex $\sigma_\mathcal{P} * \sigma_\mathcal{Q}$ in $\mathcal{T}_{\mathcal{P}*\mathcal{Q}}$. This explicitly verifies that $\mathcal{T}_{\mathcal{P}*\mathcal{Q}}$ is a flag complex.

Therefore, the join $\mathcal{P} * \mathcal{Q}$ admits a quadratic triangulation.
\end{proof}

\section{Concluding remarks}\label{Section-Seven-CR}

This paper mainly studies the Ehrhart properties of the join of two lattice polytopes. This solves an open problem on the website of the American Institute of Mathematics \cite{AimPl-Website}.

As a byproduct, we discuss the necessary and sufficient condition for the Cartesian product of two Gorenstein polytopes to remain a Gorenstein polytope.
As an open problem mentioned above (also see \cite{Ferroni23}), and one that the authors are particularly concerned with, is:
If $h^*_{\mathcal{P}}$ and $h^*_{\mathcal{Q}}$ are log-concave, is necessarily  $h^*_{\mathcal{P}\times \mathcal{Q}}$ log-concave too?

Regarding the integer decomposition property (IDP) of polytopes, we would like to propose the following problem.
\begin{prob}
For any given polytope $\mathcal{P}$, devise an efficient algorithm to determine whether $\mathcal{P}$ possesses the IDP.
\end{prob}
The design of this algorithm may be helpful for studying the following question (see \cite[Conjecture 1.1]{Ferroni23}, \cite[Question 1.1]{Scheper-Langen}): whether a lattice polytope with the IDP property is $h^*$-unimodal.
A method to determine the IDP for a certain class of extremely special polytopes is given in reference \cite{Baum-Trotter}.
Braun, Davis, and Solus \cite{Braun-Davis-Solus} investigated the IDP property of a certain class of simplices.

%\noindent\textbf{Ethical approval:} Not applicable.

%\noindent\textbf{Conflict interest:} None.

%\noindent\textbf{Declaration of competing interest:} The authors declared that they had no conflicts of interest with respect to %their authorship or the publication of this article.

%\noindent\textbf{Data availability:} Data availability is not applicable to this article as no new data were created or analyzed in %this study.

%\noindent\textbf{Funding:} This work was partially supported by the National Natural Science Foundation of China***********.

\noindent
{\small \textbf{Acknowledgments:}}
%The authors would like to thank the anonymous referees for their valuable suggestions for improving the presentation.
We are very grateful to Mingzhi Zhang for discussions on the IDP property of polytopes, and in particular for recommending the reference \cite{Baum-Trotter} to us.
This work was partially supported by the National Natural Science Foundation of China [12571355].

\end{document}